\documentclass[a4, 12pt, reqno]{article}
\usepackage{amscd,enumitem, amsmath, amssymb, amsfonts, amsthm}
\usepackage[arrow, matrix]{xy}
\usepackage{graphicx}
\usepackage{color}
\usepackage[utf8]{inputenc}
\usepackage[english]{babel}
\theoremstyle{plain}
\newtheorem{Theorem}{Theorem}[section]
\newtheorem{Lemma}{Lemma}[section]

\newtheorem{Corollary}{Corollary}[section]
\theoremstyle{definition}

\begin{document}

\title{ Extremal results on the Mostar index of trees with fixed parameters }

\author{Fazal Hayat\footnote{ E-mail: fhayatmaths@gmail.com} ,  Shou-Jun Xu\footnote{Corresponding author E-mail:  shjxu@lzu.edu.cn}\\
School of Mathematics and Statistics,  Gansu Center for Applied Mathematics,\\
 Lanzhou University,  Lanzhou 730000,  P.R. China }

 \date{}
\maketitle

\begin{abstract}
For a graph $G$, the Mostar index of $G$ is the sum of  $|n_u(e)$ - $n_v(e)|$ over all edges $e=uv$ of $G$,  where $n_u(e)$ denotes the number of vertices of $G$ that have a smaller distance in $G$ to $u$ than to $v$, and analogously for $n_v(e)$.
We determine all the graphs that maximize and minimize the  Mostar index respectively over all trees in terms of some fixed parameters like the number of odd vertices,   the number of vertices of degree two, and the number of pendent paths of fixed length.     
\end{abstract}

\section{Introduction}
Let $G$ be a simple connected  graph  with vertex set $V(G)$ and edge set $E(G)$.    The distance between $u$ and $v$ in $G$ is the least length of the path connecting $u$ and $v$ and is denoted by $d_G(u,v)$.

A molecular graph is a simple graph such that its vertices correspond to the atoms and the edges to the bonds of a molecule.   A topological index of $G$ is a real number related to $G$.  They are widely used for characterizing molecular graphs, establishing relationships between the structure and properties of molecules, predicting the biological activity of chemical compounds, and making their chemical applications.

Dosli\'c et al. \cite{DoM} introduced a bond-additive structural invariant as a quantitative refinement of the distance non-balancedness and also a measure of peripherality in graphs, named the Mostar index. For   graph $G$, the Mostar index is defined as
\[
 Mo(G)=\sum_{e=uv \in E(G)}\psi(uv),
\]
where $\psi(uv)=|n_u - n_v|$, and $n_u$ denotes the number of vertices of $G$ closer to $u$ than to $v$ and $n_v$ is the number of vertices closer to $v$ than to $u$.

 Finding extremal (maximum and minimum) graphs either for the whole class of graphs of certain fixed order or within  some sub-classes of these graphs is an important line of research  and has received much attention. For example, Dosli\'c et al. \cite{DoM} determined all the graphs that maximize and minimize the Mostar index among all trees and unicyclic graphs. Tepeh \cite{Te} obtained the graphs having the smallest and largest  Mostar index over all bicyclic graphs. Hayat and Zhou \cite{HZ} determined  the $n$-vertex cacti with the largest Mostar index, and  obtained a sharp upper bound for the Mostar index among cacti of fixed order and cycles, and characterized the extremal graphs. In \cite{HZ1}, Hayat and Zhou identified all the graphs that minimize and/or maximize the Mostar index among all trees with fixed parameters like the maximum degree, the diameter, and the number of pendent vertices. Hayat and Xu \cite{HX} investigated the lower bound on the Mostar index in tricyclic graphs, and characterized the graphs that attain the bound.
Deng and Li \cite{DL} obtained those trees with a fixed degree sequence having the largest Mostar index. In \cite{DL1}, Deng and Li studied the extremal problem for the Mostar  index among trees with a given number of segments sequence. In \cite{DL2}, Deng and Li  determined the greatest Mostar index over all chemical trees of fixed order, and the least Mostar index among trees with fixed order and diameter.  Liu and Deng \cite{LD} obtained the maximum Mostar index among unicyclic graphs with a given diameter.
More related work about the Mostar index can be found in \cite{ACT, AD, AXK, DL3, GXD, HLM, MPR,  XZT, XZT2} and the references cited therein.

To have a full understanding of the relationship between the  Mostar index and the structural properties of the graphs. In this paper, we further consider
the Mostar index among trees in terms of some other fixed parameters, and more precisely, we identify all the graphs that maximize and minimize the Mostar index among $n$-vertex trees with fixed parameters like  the number of odd vertices,   the number of vertices of degree two, and the number of pendent paths of fixed length.

\section{Preliminaries}

For  $v \in V(G)$,  $N_G(v)$ denotes the set of vertices that are adjacent to $v$ in $G$. The degree of $v \in V(G)$, denoted by $d_G(v)$,  is the cardinality of  $N_G(v)$.
Suppose $e$ is an edge and $x,y$ are two non-adjacent vertices of $G$. Then $G - e$ is the subgraph of $G$ obtained by deleting the edge $e$ and $G + xy$ is a graph obtained from $G$ by adding an edge connecting $x, y$. The graph formed from $G$ by deleting a vertex $v \in V(G)$ (and its incident edges) is denoted by $G - v$. The diameter of $G$ is the largest distance between any two vertices of $G$.
A path $P = u_0u_1 \cdots u_r$ with $r \geq 1$ in a graph $G$ is called a pendent path of length $r$ at $u_r$, if $d_G(u_0) = 1$, $d_G(u_r) \geq 2$, and if $r \geq 2$, $d_G(u_i) = 2$ for $1\leq i \leq r-1$. Particularly, a pendent path of length one at $v $ is just a pendent edge at $v$.  A pendent vertex is a vertex of degree one.  A vertex is called an odd (resp. even) vertex if its degree is odd (resp. even). A vertex in a tree with a degree of at least three is known as a branch vertex. A vertex is called an internal vertex if its degree is of at least 2. A tree whose all internal vertices have degree of at least three is called a series-reduced tree or a homeomorphically irreducible tree \cite{Has, HP}. By $S_n$ and $P_n$ we denote the star and path on $n$ vertices, respectively.
A caterpillar is a tree for which there exists some fixed path with every vertex either on the path or adjacent to a vertex on the path (the path is called the spine of the caterpillar). Note that $S_n$ and $P_n$ are caterpillars.

\begin{Theorem}\cite{DoM}\label{star}
Among all trees of order $n$,  $ S_n$ (resp. $ P_n$ ) is the unique tree with maximum (resp. minimum) Mostar index.
\end{Theorem}

Let $G$ be a graph with a cut edge $e=uv$, and let $G/e$ be the graph obtained from $G$ by contracting the edge $e$ into a new vertex $w$ such that it is adjacent to each vertex in $N_G(u)\cup N_G(v)\setminus\{u,v\}$ and then attaching  a pendent edge at $w$.

\begin{Lemma}\cite{DoM}\label{cut}
 Let $G$ be a graph with a cut edge $e$. If $e$ is not a pendent edge, then $ Mo(G/e)> Mo(G)$.
\end{Lemma}

Let $G$ be a nontrivial graph with $u \in V(G)$. For two nonnegative integers $\ell$ and  $m$, let $G_{u;\ell, m}$ be the graph obtained from $G$ by attaching two pendent paths of length $\ell$ and $m$, respectively  at $u$. In particular, $G_{u;0,0}=G$ and $G_{u;\ell, 0}$ is obtained from $G$  by attaching a pendent path of length $\ell$.

\begin{Lemma}\cite{DoM}\label{333}
Let $G$ be a nontrivial  graph with $u \in V(G)$. If $\ell\geq m\geq 1$, then  $Mo(G_{u;\ell,m}) > Mo(G_{u;\ell+1,m-1})$.
\end{Lemma}

\begin{Lemma}\cite{HZ}\label{111}
Let $G$ be a graph of order $n$ with a cut edge $e=uv$.  Then $\psi_G(e)\leq(n-2)$  with equality if and only if  $e=uv$ is a pendent edge.
\end{Lemma}

\begin{Lemma}\cite{HZ1}\label{222}
Let $T$ be a tree with $x,y\in V(T)$. Suppose that  $x_1, \dots , x_s$ and $y_1, \dots , y_t$ are pendent vertices adjacent to $x$ and $y$, respectively.
Let $x'$ and $y'$ be respectively the neighbors of $x$ and $y$ in  the path connecting $x$ and $y$. Let
  $T'= T- \{xx_i: i=1,\dots, s\} + \{x'x_i:i=1,\dots, s\}$, and
  $T''= T- \{yy_i: i=1,\dots, t\} + \{y'y_i:i=1,\dots, t\}$.
Then $Mo(T)<\max\{ Mo(T'), Mo(T'')\}$.
\end{Lemma}

For $3\le r\le n-2$,  let $S_{n,r}$ be the tree consisting of $r$ pendent paths of almost equal lengths i.e., $t$ pendent paths of length $\lfloor \frac{n-1}{r}\rfloor+1$ and $r -t $ pendent paths of length $\lfloor \frac{n-1}{r}\rfloor$ at a common vertex with $t= n-1-r\lfloor \frac{n-1}{r}\rfloor$. Particularly, $S_{n,n-1}=S_n$, and $S_{n,2}= P_n$.

\begin{Theorem}\cite{HZ1} \label{100}
Let $T$ be a tree of order $n$ with $r$ pendent vertices, where $3\le r\le n-2$, with maximum Mostar index. Then  $T \cong S_{n,r}$.
\end{Theorem}

\begin{Corollary}\label{cr} If $2 \leq r \leq n-2$, then $Mo(S_{n,r})<Mo(S_{n,r+1})$.
\end{Corollary}

\begin{proof}
Let $u$ be a vertex in $S_{n,r}$ with degree $r$ and $uv$ be an edge in the longest path of $S_{n,r}$, where if $r = 2$, $u$ is any vertex of degree two and $uv$ is in the longer sub-path with a terminal vertex $u$. By Lemma \ref{cut}, $ Mo(S_{n,r}/uv)> Mo(S_{n,r})$. Note that $S_{n,r}/uv$ is a tree  of order $n$ with $r+1$ pendent vertices.  By Theorem \ref{100}, $Mo(S_{n,r}/uv )\leq Mo(S_{n,r+1})$. It follows that $Mo(S_{n,r})<Mo(S_{n,r+1})$.
\end{proof}

Let $T_z(d_{i_1}, d_{i_2},...,d_{i_z})$ be the caterpillar of order $n$  such that $ v_1v_2 \cdots v_z$ is the spine of  $T_z(d_{i_1}, d_{i_2}, \dots ,d_{i_z})$, where the degree of $v_j$  is $d_{i_j}$ with $j = 1 , 2 , \dots , z$.

%Let $\mathcal{T}_d$ be the set of all trees with degree sequence  $(d_1 , d_2, ..., d_z)$, where $d_1\geq d_2\geq...\geq d_z \geq 2, z\geq 2$.
\begin{Lemma}\cite{DL}\label{imp}
Let $T$ be a tree with degree sequence  $(d_1 , d_2, ..., d_z)$, where $d_1\geq d_2\geq \cdots \geq d_z \geq 2, z\geq 2$ such that $T$  minimizes the Mostar index. Then $T \cong T_z(d_{i_1}, d_{i_2},\dots,d_{i_z})$, where $d_{i_1}\geq d_{i_2}\geq,\dots,\geq d_{i_s}$ and $d_{i_{s+1}} \leq d_{i_{s+2}} \leq, \dots ,\leq d_{i_z}$ for some $s \in \{1,2,\dots,z-1\}$.
\end{Lemma}

\begin{Lemma}\label{0}
Let $T$ be a tree with a longest path  $P_r= v_0v_1 \cdots v_r$.  Let  $ \{v_{i-1}, v_{i+1}, u_1, \dots , u_t\}$, where $ t\geq 1$ be the neighbours of a vertex  $v_i \in P_r$  ($1 \leq i \leq r-1 $). Suppose that $n_0$ (resp. $n_r$) be the order of the components of $T-E(P)$ containing $v_0$ (resp. $v_r$).  Let $T'=T-\{v_iu_j: 1\leq j \leq c\}+\{v_0u_j: 1\leq j \leq c\}$, where $1\leq c \leq t$. If $n_r  \geq n_0 + d_T(v_i, v_0)$, then $ Mo(T') < Mo(T) $.
\end{Lemma}

\begin{proof}
Note that $n_r  \geq n_0 + d_T(v_i, v_0)$, we have $n_0 + d_T(v_i, v_0)\le \frac{n}{2}$.
Let $s=d_T(v_i,v_0)$ and let $z_0\dots z_s$   be the path from $v_i$ to $v_0$, where $z_0=v_i$ and $z_s=v_0$. Let $n_c'$ be the  number of vertices of the components of $T-v_i$ containing one of $u_1,\dots, u_c$.
For $k=0, \dots, s$, we have $n_0+s-k<\min\{n_0+n_c'+s-k, n-(n_0+n_c'+s-k)\}\le \frac{n}{2}$, implying that
$\psi_T(z_{k-1}z_k)=|(n_0+s-k)-(n-(n_0+s-k))|>|n_0+n_c'+s-k-(n-(n_0+n_c'+s-k))|=\psi_{T'}(z_{k-1}z_k)$. Therefore,
$Mo(T)-Mo(T')=\sum_{k=1}^s\psi_T(z_{k-1}z_k)- \sum_{k=1}^s\psi_{T'}(z_{k-1}z_k)>0$, i.e., $Mo(T)>Mo(T')$.
\end{proof}

For $a,b \geq 0$ with $2(a + b)\leq n-1$, let $C_{n, a,b}$ be the tree obtained from the path $P_{n-a-b} =v_1v_2 \dots v_{n-a-b}$ by attaching a pendent vertex $u_i$ to $v_i$ for each $i$ with $2 \leq i \leq a +1$ and each $i$ with $n-a-2b \leq i \leq n-a-b-1$. In particular, $C_{n, 0,0}= P_n$ .

%\begin{Lemma} \cite{GAN} \label{edge}
%  Let $T$ be a tree of order $n\geq 4$. Then $Mo(T)>  Mo( C_{ n,0,1}) >  Mo(C_{ n,0,0})$.
 %\end{Lemma}

\begin{Lemma}\label{444}
  Let $a$ and $b$ be positive integers with $a\geq b+2$ and $2(a+b)\leq n-3$. Then, $Mo(C_{n, a-1, b+1}) < Mo(C_{n, a, b})$.
 \end{Lemma}

\begin{proof}
Let $P_{1,n-a-b}$ be the path from $v_1$ to $v_{n-a-b}$. Let $T=C_{n, a, b}$, and $T'= T- v_{a+1}u_{a+1} + v_{n-a-2b-1}u_{a+1}$. Then,   $T' \cong C_{n, a-1, b+1}$.
From  $T$ and $T'$,  one has $\psi_T(e)=\psi_{T'} (e)$ for each  $e \in T \setminus E(P_{a+1, n-a-2b-2})$. If $2a\le \frac{n}{2}$, then $2b+2 <2a \le \frac{n}{2}$, and otherwise, $2b+2  < n-2a < \frac{n}{2}$. Then $Mo(T)-Mo(T') = |2b+2-(n-2b-2)| - |2a-(n-2a)|  > 0$.   Thus $Mo(T') < Mo(T)$, implying that $Mo(C_{n, a-1, b+1}) < Mo(C_{n, a, b})$.
\end{proof}

\section{Mostar index and the number of odd vertices}

For integers $n$ and $k$ with  $ 1 \leq k\leq \lfloor \frac{n}{2}\rfloor$, let $\mathcal{O}(n,k)$ be the set of trees of order $n$ with $2k$ odd vertices.

\begin{Theorem}\label{max}
Let $T$ be in $\mathcal{O}(n,k)$, where  $ 1 \leq k\leq \lfloor \frac{n}{2}\rfloor$ with maximum Mostar index. Then $T\cong S_{n,2k}$, where if $k = \frac{n}{2}$, then $S_{n,n}= S_n$.
\end{Theorem}
\begin{proof}
The  case $k=1$ is trivial as in this case, $T \cong S_{n,2}= P_n$.

If  $k= \lfloor \frac{n}{2}\rfloor$, then $  S_{n, 2\lfloor \frac{n}{2}\rfloor}= S_n \in \mathcal{O}(n,k)$,  and so the result follows by Theorem \ref{star}.

Suppose that $ 2 \leq k< \lfloor \frac{n}{2}\rfloor$, i.e., $ 4 \leq 2k\leq  n-2$. Let $r$ be the number of pendent vertices of $T$, then $r \leq 2k$. By Theorem \ref{100} and  Corollary \ref{cr}, $Mo(T)\leq Mo(S_{n,r}) \leq Mo(S_{n,2k})$ with equalities if and only if $T \cong S_{n,r}$ and $r=2k$, i.e., $T \cong S_{n,2k}$.
\end{proof}

\begin{Theorem}\label{min}
Let $T$ be in $\mathcal{O}(n,k)$, where  $ 1 \leq k\leq \lfloor \frac{n}{2}\rfloor$ with minimum Mostar index. Then $T\cong C_{n,\lceil\frac{k-1}{2}\rceil, \lfloor \frac{k-1}{2} \rfloor}$.
\end{Theorem}

\begin{proof}
 Let $T\in \mathcal{O}(n,k)$  such that $Mo(T)$ is as small  as possible. Suppose that $(d_1 , d_2, ..., d_z)$ is the degree sequence of $T$, where $d_1\geq d_2\geq...\geq d_z \geq 2$,   $z\geq 2$. Then by Lemma  \ref{imp}, $T \cong T_z(d_{i_1}, d_{i_2},...,d_{i_z})$, where $d_{i_1}\geq d_{i_2}\geq,...,\geq d_{i_s}$ and $d_{i_{s+1}} \leq d_{i_{s+2}} \leq, ...,\leq d_{i_z}$ for some $s \in \{1,2,...,z-1\}$.

 The case  $k=1$ is trivial.
Suppose that $k\geq 2$.
 Since $T$ is a caterpillar, let $P= v_0v_1 \cdots v_rv_{r+1}$ be the spine of $T$. We claim that the maximum odd degree of $T$ is exactly 3. Otherwise,  there exist a vertex $v_i$ ($1\leq i \leq r$) such that $d_T(v_i)= 2t+1$, where $t\geq 2$. Let $u_1,...,u_{2t-1}$ be the neighbors of $v_i$ outside $P$ in $T$. Let $T'$ be the tree obtained from $T$ by deleting the pendent vertex $u_{2t-1}$ and the edges $v_iu_1, v_iu_2,..., v_iu_{2t-2}$, splitting $v_i$ into two adjacent vertices $v'_i$ and $v''_i$, joining $v_iu_1, \dots , v_iu_{2t-3}$ to  $v'_i$ and  $u_{2t-2}$ to $v''_i$.
 Obviously, $T'\in \mathcal{O}(n,k)$.
We denote by $e_1, e_2,...e_{r-1}$   (resp. $f_1, f_2,...f_r$  )  the consecutive edges in the spine $P=v_1 \cdots v_i \cdots v_r$ of $T$ (resp. $P= v_1 \cdots v'_iv''_i \cdots v_r$ of $T'$). Note that the number of pendent edges in $T$ is $|E(T)|-(r-1)= n-r$,  and the number of pendent edges in $T'$ is $|E(T')|-r = n-1-r$.

From $T$ and $T'$,  one has $\psi_T(e_j)=\psi_{T'} (f_j)$ for any  $j \in \{1,2,\dots,i-1\}$, and  $\psi_T(e_{\ell})=\psi_{T'} (f_{\ell+1})$ for any  $\ell \in \{i,i+1,...,r-1\}$. Then $Mo(T)-Mo(T') =  (n-2)(n-r)-((n-2)(n-r-1) + \psi_{T'} (f_i) ) = (n-2) - \psi_{T'} (f_i)> 0$.  By Lemma \ref{111},  we have $Mo(T')<Mo(T)$, a contradiction. Thus, the maximum odd degree of $T$ is exactly 3.

If  $k=\frac{n}{2}$ i.e., $2k=n$, then by Proposition \ref{imp},  $T\cong C_{n,\lceil\frac{k-1}{2}\rceil, \lfloor \frac{k-1}{2} \rfloor}$.

Suppose that $k < \frac{n}{2}$.   Then there is at least one even vertex in $T$. We claim that all the even vertices have a degree exactly 2. Otherwise,  for some $v_i$   $(1 \leq i \leq r)$, we have $d_T(v_i)= 2t$ such that  $N_T(v_i) = \{v_{i-1}, v_{i+1}, u_1, \dots , u_{2t-2}\}$, where $ t\geq 2$. Suppose that $n_1$ (resp. $n_2$) be the order of the components of $T-v_i$ containing $v_0$ (resp. $v_{r+1}$). Assume that  $n_1 \leq n_2$. Let $T''=T-\{v_iu_3+v_0u_3\}$, where $v_0$ is the pendent vertex in $T$. Note that in $T''$ the degree of $v_i$ is odd and the degree of $v_0$ is even, implying that $T''\in \mathcal{O}(n,k)$. By Lemma \ref{0}, $Mo(T'')<Mo(T)$, a contradiction. Thus, all the even vertices in $T$ have a degree exactly 2.

Consequently, $T$ is a tree with exactly $n-2k$ vertices of a degree 2.  Note that $T$ is a caterpillar, therefore the number of pendent vertices in $T $ is $k+2$. Let $t$ be the numbers of  vertices of a degree 3 in $T$. The relation $\sum_{v \in V(G)}d_T(v)=2|E(T)|=2(n-1)$ gives that $k+2+2(n-2k)+3t=2n-2$,  so $t=k-1$. Therefore, the degree sequence of $T$ is $(\underbrace{3,...,3}_\text{k-1}, \underbrace{2,...,2}_\text{n-2k}, \underbrace{1,...,1}_\text{k+2})$.  By Lemma \ref{imp}, $T \cong C_{n, a, b}$, for some $a,b$ with $a+b= k-1$. By Lemma \ref{444}, $T\cong C_{n,\lceil\frac{k-1}{2}\rceil, \lfloor \frac{k-1}{2} \rfloor}$.
\end{proof}

 Let $T$ be a tree of order $n$ with $k$ branch vertices. Let $s$ be the number of pendent vertices in $T$. Then $ s + 2(n-s-k)+3k \leq 2(n-1)$, i.e., $s \geq k + 2$. This implies that $2k +2 \leq k + s \leq n$, and thus, $k \leq \frac{n}{2}-1$

\begin{Corollary}\label{cr2} Let $T$ be a tree of order $n$ with $k$ branch vertices, where  $ 0 \leq k\leq  \frac{n}{2}-1$ with minimum Mostar index. Then $T\cong C_{n,\lceil\frac{k}{2}\rceil, \lfloor \frac{k}{2} \rfloor}$.
\end{Corollary}

\begin{proof}
Let $T$ be  a tree of order $n$ with $k$ branch vertices  such that $Mo(T)$ is as small  as possible.
If  $k=0$, then $T \cong P_n$ and it is trivial.

 Suppose that $k\geq 1$.
 Let $(d_1 , d_2, \dots , d_z)$ be the degree sequence of $T$, where $d_1\geq d_2\geq \cdots\geq d_z \geq 2, z\geq 2$. Then by Lemma  \ref{imp}, $T \cong T_z(d_{i_1}, d_{i_2},...,d_{i_z})$, where $d_{i_1}\geq d_{i_2}\geq,...,\geq d_{i_s}$ and $d_{i_{s+1}} \leq d_{i_{s+2}} \leq, ...,\leq d_{i_z}$ for some $s \in \{1,2,...,z-1\}$.

Since $T$ is a caterpillar, let $P=v_0v_1 \dots v_r$ be the spine of $T$. Let $\Delta$ be the maximum degree of $T$. We claim that $\Delta =3$. Suppose that $\Delta \geq 4$, then for some $v_i \in V(T)$, where  $(1 \leq i \leq r-1 $) such that  $N_T(v_i) = \{v_{i-1}, v_{i+1}, u_1, \dots , u_{\Delta -2}\}$. Suppose that $n_1$ (resp. $n_2$) be the order of the components of $T-v_i$ containing $v_0$ (resp. $v_r$). Without loss of generality, we assume that  $n_1 \leq n_2$. Let $T'=T-\{v_iu_1+v_0u_1\}$, where $v_0$ is a pendent vertex of $T$. Obviously, $T'$ is a tree of order $n$ with $k$ branch vertices. By Lemma \ref{0}, $Mo(T')<Mo(T)$, a contradiction. Thus, $\Delta = 3$.  Let $s$ be the number of pendent vertices in $T$. As $ s + 2(n-s-k)+3k = 2(n-1)$, we have $s = k+2$,  and thus $T$ is a tree of order $n$ with $2k +2$ odd vertices. By Theorem \ref{min}, we have $T\cong C_{n,\lceil\frac{k}{2}\rceil, \lfloor \frac{k}{2} \rfloor}$.
\end{proof}

Let $\mathcal{O}_{2n}$ be the set of all trees with $2n$ vertices, which are all odd. By Theorem \ref{star} and Theorem \ref{min}, we have the following result.

\begin{Theorem}\label{odd}
Let $T$ be in $\mathcal{O}_{2n}$, with maximum (resp. minimum) Mostar index. Then  $T\cong S_{2n}$ (resp. $T\cong C_{2n, 0, n-1}$).
\end{Theorem}

\section{Mostar index and the number of vertices of degree two}

For integers $n$ and $t$ with  $ 0 \leq t\leq n-2$, let $\mathcal{E}(n,t)$ be the set of trees of order $n$ with $t$ vertices of degree two.
Note that $\mathcal{E}(n,n-2)=\{P_n\}$ and $\mathcal{E}(n,n-3)=\emptyset$. So we only consider the trees in $\mathcal{E}(n,t)$ with  $ 0 \leq t\leq n-4$.

\begin{Theorem}\label{}
Let $T$ be in $\mathcal{E}(n,t)$, where $ 0 \leq t\leq n-4$  with  maximum Mostar index. Then  $T\cong S_{n, n-t-1}$.
\end{Theorem}

\begin{proof}
 Let $T \in \mathcal{E}(n,t)$  such that $Mo(T)$ is as large  as possible.

If  $t=0$, then $S_{n, n-1}= S_n \in \mathcal{E}(n,t)$, so the result follows by Theorem \ref{star}.

Suppose that $t\geq 1$. We claim that $T$ contains exactly one vertex having degree at least 3. Otherwise, $T$ contains two vertices, say $u, v$ having  degree at least 3.
Suppose that  $u_1, \dots , u_x$ and $v_1, \dots , v_y$ be the pendent vertices adjacent to $u$ and $v$, respectively.
Let $u'$ (resp. $v'$) be  the neighbor of $u$ (resp. $v$) in  the path connecting $u$ and $v$.
 Let $T'= T- \{uu_i: i=1,\dots, x\} + \{u'u_i:i=1,\dots, x\}$, and $T''= T- \{vv_i: i=1,\dots, y\} + \{v'v_i:i=1,\dots, y\}$. Then by Lemma \ref{222}, $Mo(T)<\max\{ Mo(T'), Mo(T'')\}$, a contradiction. Thus, $T$ contains exactly one vertex having degree at least 3. Let $z$ be the number of pendent vertices of $T$, then $z \le n-t-1$. By Theorem \ref{100} and  Corollary \ref{cr}, we have $Mo(T)\leq Mo(S_{n,z}) \leq Mo(S_{n,n-t-1})$ with equalities if and only if $T \cong S_{n,z}$ and $z=n-t-1$, i.e., $T \cong S_{n,n-t-1}$.
\end{proof}

For $a,b \geq 0$ with $2(a + b)\leq n-5$, let $F_{n, a, b}$ be the tree obtained from the path $P_{n-a-b-2}=  v_1v_2 \cdots v_{n-a-b-2}$ by attaching a pendent vertex $u_1$ to $v_2$, and then attaching pendent vertex $u_i$ to $v_i$ for each $i$ with $2 \leq i \leq a +2$ and each $i$ with $n-a-2b-2 \leq i \leq n-a-b-3$.

\begin{Lemma}\label{555}
 For $a\geq b \geq 1$ and $2(a+b) < n-5$, we have  $Mo(F_{n, a-1, b+1}) < Mo(F_{n, a, b})$.
 \end{Lemma}

 \begin{proof}
Let $P_{1,n-a-b-2}$ be the path from $v_1$ to $v_{n-a-b-2}$. Let $T=F_{n, a, b}$, and $T'= T- v_{a+2}u_{a+2} + v_{n-a-2b-3}u_{a+2}$. Then   $T' \cong F_{n, a-1, b+1}$.
From  $T$ and $T'$,  one has $\psi_T(e)=\psi_{T'} (e)$ for each  $e \in T \setminus E(P_{a+2, n-a-2b-4})$.  If $2a+3\le \frac{n}{2}$, then $2b+2 <2a+3 \le \frac{n}{2}$, and otherwise, $2b+2<n-2a-3<\frac{n}{2}$.  Then $Mo(T)-Mo(T') = |2b+2-(n-2b-2)| - |2a+3 -(n-2a-3)|  > 0$.   Thus $Mo(T') < Mo(T)$, implying that $Mo(F_{n, a-1, b+1}) < Mo(F_{n, a, b})$.
\end{proof}

\begin{Theorem}\label{two}
Let $T$ be in $\mathcal{E}(n,t)$, where $ 0 \leq t\leq n-4$ with  minimum Mostar index.
\begin{enumerate}[label=(\roman*) ]
\item If $n-t=5$, then $T \cong F_{n, 0, 0}$. If $n-t$ is odd and at least $7$, then $T\cong F_{n, \lceil\frac{n-t-5}{4}\rceil-1, \lfloor\frac{n-t-5}{4}\rfloor+1 }$.
\item If $n-t$ is even, then $T\cong C_{n, \lceil\frac{n-t}{4}-\frac{1}{2}\rceil, \lfloor\frac{n-t}{4}-\frac{1}{2}\rfloor} $.
\end{enumerate}
\end{Theorem}

\begin{proof}
Let $T$ be in $\mathcal{E}(n,t)$  such that $Mo(T)$ is as small  as possible.
Suppose that $(d_1 , d_2, ..., d_z)$ be the degree sequence of $T$, where $d_1\geq d_2\geq...\geq d_z \geq 2, z\geq 2$.By lemma  \ref{imp}, $T \cong T_z(d_{i_1}, d_{i_2},...,d_{i_z})$, where $d_{i_1}\geq d_{i_2}\geq,...,\geq d_{i_s}$ and $d_{i_{s+1}} \leq d_{i_{s+2}} \leq, ...,\leq d_{i_z}$ for some $s \in \{1,2,...,z-1\}$.

Since $T$ is a caterpillar, let $P= v_0v_1 \cdots v_r$ be the spine of $T$.
Let $\Delta $ be the maximum degree of $T$. We claim that $\Delta \leq 4$. Suppose that $\Delta \geq 5$, then for some $v_i \in V(T)$, where  $(1 \leq i \leq r-1 $) such that  $N_T(v_i) = \{v_{i-1}, v_{i+1}, u_1, \dots , u_{\Delta -2}\}$. Suppose that $n_1$ (resp. $n_2$) be the order of the components of $T-v_i$ containing $v_0$ (resp. $v_r$). Assume that  $n_1 \leq n_2$. Let $T'=T-\{v_iu_j : 2\le j \le \Delta-2 \}+\{v_0u_j: 2\le j \le \Delta-2 \}$, where $v_0$ is a pendent vertex of $T$. Obviously, $T' \in \mathcal{E}(n,t)$. By Lemma \ref{0}, $Mo(T')<Mo(T)$, a contradiction. Thus, $\Delta \leq 4$.

Further, we claim that $T$ has at most one vertex of degree 4.
Suppose that $T$ contains at least two vertices of degree 4, say $x,y$. Let $n_x$ (resp. $n_y$) be the order of the component of $T-E(P)$ containing $x$ (resp. $y$).
Assume that $n_x \ge n_y + d_T(x,y)$. Then $n_y + d_T(x,y)\le \frac{n}{2}$.
Let $u$  be the neighbor of $x$ outside $P$ in $T$. Let $T''=T-xu+yu$. Obviously, $T'' \in \mathcal{E}(n,t)$. By Lemma \ref{0}, $Mo(T')<Mo(T)$, a contradiction.  Thus, $T$ contains at most one vertex of degree 4.

Let $d_i$ be the number of vertices of degree $i$ in $T$, where $1 \leq i\leq 4$. Note that $d_1+d_2+d_3+d_4=n$ and $d_1+2d_2+3d_3+4d_4=2(n-1)$. Since $d_2=t$ and $d_4 \leq 1$.

\noindent {\bf Case 1.} $n-t$ is odd.

From the above relation, we have $d_1= \frac{n-t+3}{2},  d_3= \frac{n-t-5}{2}$, and $d_4 =1$. If $n-t=5$, then the degree sequence of $T$ is
$(4, \underbrace{2,...,2}_\text{t}, \underbrace{1,...,1}_\text{n-t-1})$.
By Lmma \ref{imp}, $T \cong F_{n, 0, 0}$. If $n-t\geq 7$, then the degree sequence of $T$ is
$(4, \underbrace{3,...,3}_\frac{n-t-5}{2}, \underbrace{2,...,2}_\text{t}, \underbrace{1,...,1}_\frac{n-t+3}{2})$.
By Lemma \ref{imp}, $T \cong F_{n, a, b}$,  for some $a,b$ with $a+b= \frac{n-t-5}{2}$. By Lemma \ref{555}, $T\cong F_{n, \lceil\frac{n-t-5}{4}\rceil-1, \lfloor\frac{n-t-5}{4}\rfloor+1 }$.

\noindent {\bf Case 2.} $n-t$ is even.

From the above relation, we have $d_1= \frac{n-t+2}{2}$,   $d_3= \frac{n-t-2}{2}$, and $d_4 =0$.  Since $d_4 =0$, then the maximum degree of $T$ is 3, and the degree sequence of $T$ is ($\underbrace{3,...,3}_\frac{n-t-2}{2}, \underbrace{2,...,2}_\text{t}, \underbrace{1,...,1}_\frac{n-t+2}{2})$ by Proposition \ref{imp}, $T \cong C_{n, a, b}$, for some $a,b$ with $a+b= \frac{n-t-2}{2}$. By Lemma \ref{444}, $T\cong C_{n, \lceil\frac{n-t}{4}-\frac{1}{2}\rceil, \lfloor\frac{n-t}{4}-\frac{1}{2}\rfloor} $.
\end{proof}

As an immediate consequence of Theorem \ref{two}, we obtain the following result.

\begin{Corollary}
Let $T$ be a series-reduced tree of order $n\geq 5$ with minimum Mostar index. Then  $T\cong F_{n, \lceil\frac{n-5}{4}\rceil-1, \lfloor\frac{n-5}{4}\rfloor+1 }$ (resp.  $T\cong C_{n, \lceil\frac{n-2}{4}\rceil, \lfloor\frac{n-2}{4}\rfloor} $ ) for odd $n$ (resp. even $n$ ).
\end{Corollary}

\section{Mostar index and the number of pendent paths of fixed length}

For integers $n,r$ and $k$, let $S^r_{n, k}$ be the tree consisting of $k$ pendent paths of length $r$  and $n-kr-1$ pendent paths of length one at a common vertex, where $k=1$ and  $ 2 \leq r\leq n-3$, or $k\geq 2$, $r\geq 2$ and $kr \leq n-2$.

\begin{Theorem}\label{}
Let $T$ be a tree of order $n$ with $k$ pendent paths of length $r$, where $k=1$ and $ 2 \leq r\leq n-3$, or $k\geq 2$, $r\geq 2$ and $kr \leq n-2$, with  maximum Mostar index. Then $T\cong S^r_{n, k}$.
\end{Theorem}

\begin{proof}
Let $T$ be a tree of order $n$ with $k$ pendent paths of length $r$ such that $Mo(T)$ is as large as possible.

Let $X$  be the set of vertices at which there is a pendent path of length $r$ in $T$. Obviously, the order of $X$ is at least one. We claim that $X$ contains exactly one vertex. Suppose that $|V(X)|\geq 2$, and $v_1, v_2 \in V(X)$  such that $d_T(v_1,v_2)$ is as small as possible. Let $P$ be the path connecting $v_1$ and $v_2$. If
$d_T(v_1,v_2)>1$, then each internal vertex of $P$ is of degree two. Let $n_1$ (resp. $n_2$) be the order of the component of $T-E(P)$ containing $v_1$ (resp. $v_2$).
Assume that $n_1\ge n_2$. Obviously, $n_2\le \frac{n}{2}$.
Let $u$ be the neighbor of $v_2$ in $P$  and  $w$  be one other neighbor of $v_2$. Let $n_w=n_w(v_2w|T)$.
 Let $T'=T-v_2w+uw$. Obviously,  $T'$ is a tree of order $n$ with $k$ pendent paths of length $r$. Note that $\psi_T(e)=\psi_{T'}(e)$ for $e\in E(T)\setminus \{uv_2, v_2w\}=E(T')\setminus\{uv_2, uw\}$ and $\psi_T(v_2w)=\psi_{T'}(uw)$. Then
 \[
 Mo(T)-Mo(T')=\psi_T(uv_2)-\psi_{T'}(uv_2)=|n_2-(n-n_2)|-|(n_2-n_w)-n-(n_2-n_w)|<0,
 \]
implying that $Mo(T)<Mo(T')$, a contradiction. Thus, $T$ consists of $k$ pendent paths of length $r$  at a common vertex. Let $X= \{z\}$

Note that if $k=1$ (resp. $k\geq 2$ ) then there are two (resp. one) neighbours of $z$ not lying on the pendent path of length $r$. We claim that all neighbours of $z$ not lying on the pendent path of length $r$ are pendent. Let $s$ be any such neighbours of $z$. Suppose that $d_T(s) \geq 2$.
Let $T/zs$ be the tree obtained from $T$ by contracting the edge $e=zs$ into a new vertex, say $u$ (such that it is adjacent to each vertex in $N_T(z)\cup N_T(s)\setminus\{z,s\}$ and then attaching  a pendent edge at $u$. Obviously, $T/zs$ is a tree of order $n$ with $k$ pendent paths of length $r$. By Lemma \ref{cut}, $Mo(T)<Mo(T/zs)$, a contradiction. Thus, $d_T(s) = 1 $.  That is,  $T\cong S^r_{n, k}$.
\end{proof}

For positive integers $n,r, a$ and $b$ with $a\geq b$ and $(a + b)r\leq n-2$, let $A^r_{n, a, b}$ be the $n-$vertex  tree obtained from the path $P_{n-a-b}= v_1v_2 \dots v_{n-a-b}$ by attaching $a$ and $b$ pendent paths of length $r$  at $v_1$ and $v_{n-a-b}$, respectively.

\begin{Lemma}\label{666}
Suppose that $a\geq b+2$ and $(a + b)r < n-1$. Then $Mo(A^r_{n, a-1, b+1}) < Mo(A^r_{n, a, b})$.
\end{Lemma}

\begin{proof}
Let $P_{1,n-a-b}$ be the path from $v_1$ to $v_{n-a-b}$. Let $T=A^r_{n, a, b}$. Let $V_1$ be all the vertices of $a$ pendent paths at $v_1$ except $v_1$ and $u_1u_2 \cdots u_rv_1$ be one pendent path at $v_1$. Let $V_2$ be all the vertices of $b$ pendent paths at $v_{n-a-b}$ except $v_{n-a-b}$ and  $w_1w_2 \cdots w_rv_{n-a-b}$ be one pendent path at $v_{n-a-b}$.
Let $T' = T-v_1u_r+v_{n-a-b}u_r$. Clearly, $T' \cong A^r_{n, a-1, b+1}$.
 From $T$ and $T'$,  one has $\psi_T(e)=\psi_{T'} (e)$ for each  $e \in T \setminus E(P_{1,n-a-b})$.

If $ar\le \frac{n}{2}$, then $(b+1)r < ar\le \frac{n}{2}$, and otherwise, $(b+1)r <  n-ar <\frac{n}{2}$.
Then $Mo(A^r_{n, a, b})-Mo(A^r_{n, a-1, b+1})=|(b+1)r-(n-br-r)|-|ar-(n-ar)|>0$.
 Thus, $Mo(A^r_{n, a, b})> Mo(A^r_{n, a-1, b+1})$.
\end{proof}

\begin{Theorem}\label{}
Let $T$ be a tree of order $n$ with $k$ pendent paths of length $r$ with  minimum Mostar index.
\begin{enumerate}[label=(\roman*) ]
\item If $k=1$ and $ 2 \leq r\leq n-3$, then $T \cong A^1_{n, 1, 2}$.
\item If $k=2$ and $ 1 \leq r\leq n-2$, then $T \cong P_n$.
\item If $k\geq 3$ and $ 1 \leq r\leq \frac{n-2}{k}$, then $T \cong A^r_{n, \lceil\frac{k}{2}\rceil, \lfloor\frac{k}{2}\rfloor}$.
\end{enumerate}
\end{Theorem}

\begin{proof}
Let $T$ be a tree of order $n$ with $k$ pendent paths of length $r$ such that $Mo(T)$ is as small as possible.

  If $k=1 $ and $ 2 \leq r\leq n-3$, then $T \cong P_n$.
As $P_n$ is $n$-vertex tree with two pendent paths of length $r$ for each $r = 2,..., n-2$ and $A^1_{n, 1, 2}$ is $n$-vertex tree with one pendent path of length $r$ for each
$r = 2,..., n-3$ it follows that  $T \cong A^1_{n, 1, 2}$. This proves Item (i).

If $k=2 $ and $ 1 \leq r\leq n-2$,  then the result follows by Theorem \ref{star}.
As $P_n$  is the unique tree with minimum  Mostar index among all $n$-vertex trees. This proves Item (ii).

Suppose that $k\geq 3$ and $ 1  \leq r\leq \frac{n-2}{k}$.
Let $X$  be the set of vertices at which there is a pendent path of length $r$ in $T$. Clearly, the order of $X$ is at least one. We claim that $X$ contains exactly two vertices,  we need to show that $|X|\leq 2 $ and $|X|\geq 2 $.

First, we show that $|X|\leq 2 $.
Suppose that $|X|\geq 3 $.  We may choose two vertices, say $u, v $  such that $d_T(u,v)$ is as large as possible. Let $P$ be the path connecting $u$ and $v$.  Let $n_u$ (resp. $n_v$) be the order of the component of $T-E(P)$ containing $u$ (resp. $v$). We may choose vertices $w$ and $c$ in $P$ such that both $d_T(u, w)$ and $d_T(c,v)$ is as small as possible.
Assume that $n_u + d_T(u,w)\ge n_v + d_T(c,v)$. Then $n_v + d_T(c,v)\le \frac{n}{2}$.
Let $x_1, \dots x_p$  be the neighbors of $c$ outside $P$ in $T$, where $p=d_T(c)-2$. Let $T'=T-\{cx_i: i=1, \dots, p\}+\{vx_i: i=1, \dots, p\}$. Obviously, $T'$ is a tree of order $n$ with $k$ pendent paths of length $r$. By Lemma \ref{0}, $Mo(T)>Mo(T')$, a contradiction. Thus, $|X|\leq 2 $.

Secondly, we show that $|X|\geq 2 $. If $r=1$, and $k < n-1$, we have $|X|\geq 2 $.
Suppose that $r \geq 2$ and $|X|=1 $.
Then $k$ pendent paths of length $r$ of $T$ are $k$ branches of $T$ at a common vertex, say $v$ which we denote by $Q_1, \dots , Q_k$.
Let $T_v= T-\cup_{i=1}^k (V(Q_i)\setminus\{v\})$.  Then $T_v$ is a nontrivial tree as $rk < n-1$.
Moreover, $T_v$ is not pendent path of length greater than or equal to $r$ at $v$, as otherwise, $T$ must have $k +1$ pendent paths of length $r$, which is impossible.

Let $X=  V(Q_1)\setminus\{v\}$.
Suppose  that there is a pendent path, say $B_1$, at $v$ in $T$, which is also
a path in $T_v$. Then the length $\ell$ of $B_{\ell}$ is less than $r$. Let $X'=  V(B_1)\setminus\{v\}$ and $G = T - X -X'$. Then $T \cong G_{v,r,l}$.  It is easy to see that $G_{v,r+l,0}$
is a tree of order $n$ with $k$ pendent paths of length $r$. By Lemma \ref{333}, we have
$Mo(T)=Mo(G_{v, r,\ell}) > Mo(G_{v, r+1, \ell-1}) > \dots > Mo(G_{v, r+ \ell, 0})$,
a contradiction. Thus,  there is no pendent path at $v$ in $T$, which is also a path in $T_v$. There must be a vertex different from $v$ with degree at least 3
in $T_v$. Choose such a vertex $u$ so that the distance to $v$ is as small as possible. Let $Q$
be the branch of $T_v$ at $v$ containing $u$. Note that there are two pendent paths
at $u$ in $T_v$. Let $s_1$ and $s_2$ be their lengths with $s_1 \le s_2 $. Then $T \cong H_{u,s_1,s_2}$, where $H$ is a subtree of $T$.
By repeating this process, we may  obtain  a tree $T^\ast$ from $T$ such that the branch $Q$
of $T_v$ at $v$ is changed into a pendent path $B_2$ at $v$ with length $s= |V(Q|-1$.
By Lemma \ref{333}, we have
$ Mo(T) > Mo(H_{v, s_1 +s_2,0})  > \dots > Mo(T^\ast)$.
Note that $Q_1, \dots ,Q_k$ have still $k$ pendent paths of length $r$ in $T^\ast$.
Let $Y=  V(B_2)\setminus\{v\}$. Let $H' = T^\ast - X - Y$. Then  $T^\ast \cong H'_{v,r,s}$.  Clearly,  $H'_{v,r+s,0}$
ia a tree of order $n$ with $k$ pendent paths of length $r$.
By Lemma \ref{333}, we have $ Mo(T^\ast) = Mo(H'_{v,r,s})  > \dots > Mo(H'_{v,r+s,0})$, a contradiction. Thus, $|X|\geq 2$.
Hence,  $|X|= 2$, as claimed.

Now we show that $T \cong A^r_{n, a, b}$ for some $a$ and $b$ with $a \geq b \geq 1$ and $a+b=k$. If $r=1$, it is trivial.
Suppose that $r \geq 2$. Assume that $X= \{s_1, s_2 \}$. Let $P$ be the path connecting $s_1$ and $s_2$. Suppose that there is a vertex of degree at least three different from $s_1$ and $s_2$, say $w$ on  $P$.
 Let $n_1$ (resp. $n_2$) be the order of the component of $T-w$ containing $s_1$ (resp. $s_2$).
Assume that $n_1 \ge n_2 + d_T(w,s_2)$.
Let $x_1, \dots ,x_p$  be the neighbors of $w$ outside $P$ in $T$. Let $T^\star=T-\{wx_i: i=1, \dots, p\}+\{s_2x_i: i=1, \dots, p\}$. Obviously, $T^\star$ is a tree of order $n$ with $k$ pendent paths of length $r$. By Lemma \ref{0}, $Mo(T) > Mo(T^\star)$, a contradiction.
 Thus, all internal vertices of $P$ (if any exist) is of degree 2.

Denote by $a$ (resp. $b$) the number of pendent paths of length $r$ at $s_1$ (resp. $s_2$). Obviously, $a+b = k$, $d_T(s_1) \geq a+1$ and $d_T(s_2) \geq b+1$.
If $d_T(s_1) \geq a+2$, then there is a branch, say $Q$ at $s_1$ not containing $s_2$, which is not a pendent path of length $r$, but, as above by Lemma \ref{333}, $Q$ can not be a pendent path and it can not have a vertex different from $s_1$ of degree at least 3. So $d_T(s_1) = a+1$. Similarly, $d_T(s_2) = b+1$. Thus, $T \cong A^r_{n, a, b}$ for some $a$ and $b$ with $a, b \geq 1$ and $a+b=k$.
Assume that $a\geq b$. If $a\geq b+2$, then by Lemma \ref{666}, $Mo(A^r_{n, a-1, b+1}) < Mo(A^r_{n, a, b})$, a contradiction. Therefore, $a-b = 0, 1$, implying that $T \cong A^r_{n, \lceil\frac{k}{2}\rceil, \lfloor\frac{k}{2}\rfloor}$.
The case $r=1$ has been proved in \cite{HZ1}. This completes the proof.
\end{proof}

\section{Concluding remarks}

In this paper, we identified all the  graphs having the largest and smallest Mostar index respectively among $n$-vertex trees with fixed parameters like the number of odd vertices, the number of vertices of degree two and the number of pendent paths of fixed length. We also identified the unique graph that maximizes and minimizes the Mostar index respectively among trees with all odd vertices. Moreover, we identified the unique  graph that minimizes the Mostar index over all trees of order $n$ with a fixed number of branch vertices and among all series-reduced trees, respectively. For future work, it will be  interesting to determine all the graphs that maximize and minimize the Mostar index over  all series-reduced trees of order $n$ with fixed parameters like the diameter, the maximum degree, and the number of pendent vertices.

%\vspace{3mm}
%
%\noindent {\bf Author Contributions:} Both the authors contributed equally to the work.
%
%
%\vspace{3mm}
%
%\noindent {\bf Conflicts of Interest:} The authors declare that they have no conflict of interest regarding the publication of this article.

\vspace{3mm}

\noindent {\bf Acknowledgement:} This work was partially supported  by the National Natural Science Foundation of China (Grant Nos. 12071194, 11571155).


\vspace{3mm}

\begin{thebibliography}{20}

\bibitem{ACT} M. Arockiaraj, J. Clement and N. Tratnik, Mostar indices of carbon nanostructures and circumscribed donut benzenoid systems. Int. J. Quantum Chem. 119 (2019) e26043.

\bibitem{AD}  A. Ali, T. Do\v{s}li\'c, Mostar index: results and perspectives.  Appl. Math. Comput. 404 (2021) 19. 126245.

\bibitem{AXK}  Y. Alizadeh, K. Xu, S. Klav\v{z}ar, On the Mostar index of trees and product graphs.  Filomat 35 (2021) 4637--4643.

\bibitem{DL}  K. Deng, S. Li, On the extremal values for the Mostar index of trees with given degree sequence. Appl. Math. Comput. 390 (2021)  125598.

\bibitem{DL1}  K. Deng, S. Li, On the extremal Mostar indices of trees with a given segment sequence. Bull. Malays. Math. Sci. Soc. 45 (2021)  593--612.

\bibitem{DL2}  K. Deng, S. Li, Chemical trees with extremal Mostar index. MATCH Commun. Math. Comput. Chem.  85  (2021) 161--180.

\bibitem{DL3}  K. Deng, S. Li, Extremal catacondensed benzenoid with respect to the Mostar index. J. Math. Chem.  58  (2020) 1437--1465.

\bibitem{DoM}  T. Do\v{s}li\'c, I.  Martinjak, R. \v{S}krekovski, S. Tipuri\'c Spu\v{z}evi\'c, I.  Zubac, Mostar index. J. Math. Chem.  56  (2018) 2995--3013.

\bibitem{GXD}  F. Gao, K. Xu, T. Do\v{s}li\'c, On the difference between Mostar index and irregularity of graphs. Bull. Malays. Math. Sci. Soc. 44 (2021)  905--926.

\bibitem{Has}  J. Haslegrave, Extremal results on average subtree density of series-reduced trees. J. Combin. Theory Ser. B. 107 (2014) 26--41.

\bibitem{HP}  F. Harary, G. Prins, The number of homeomorphically irreducible trees, and other species.  Acta Math. 101 (1959) 141--162.

\bibitem{HX}  F. Hayat, S. J. Xu, A lower bound on the Mostar index of tricyclic graphs. Filomat   (Accepted).

\bibitem{HZ}  F. Hayat, B. Zhou, On cacti with large Mostar index. Filomat  33 (2019) 4865--4873.

\bibitem{HZ1} F. Hayat, B. Zhou, On Mostar index of trees with parameters. Filomat  33 (2019) 6453--6458.

\bibitem{HLM} S. Huang, S. Li, M. Zhang, On the extremal Mostar indices of hexagonal Chains. MATCH Commun. Math. Comput. Chem.  84  (2020) 249--271.

%\bibitem{JKR} J. Jerebic, S. Klav\v{z}ar, D.F. Rall, Distance-balanced graphs. Ann. Combin. 12 (2008) 71--79.

\bibitem{LD}  G. Liu, K. Deng, The maximum Mostar indices of unicyclic graphs  with given diameter. Appl. Math. Comput. 439 (2023)  127636.

\bibitem{MPR} \v{S}. Miklavi\v{c}, J. Pardey, D. Rautenbach,  F. Werner, Maximizing the Mostar index for bipartite graphs and split graphs. Discrete Optim. 48 (2023) 100768.

%\bibitem{MS} \v{S}. Miklavi\v{c}, P. \v{S}parl, $\ell$-distance-balanced graphs. Discrete Appl. Math. 244 (2018) 143--154.

\bibitem{Te} A. Tepeh,  Extremal bicyclic graphs with respect to Mostar index.  Appl. Math. Comput.  355  (2019) 319--324.

\bibitem{XZT}  Q. Xiao, M. Zeng, Z. Tang,  The hexagonal chains with the first three maximal Mostar indices. Discrete  Appl. Math. 288  (2020) 180--191.

\bibitem{XZT2} Q. Xiao, M. Zeng, Z. Tang,  H. Deng, H. Hua,  Hexagonal chains with first three minimal Mostar indices. MATCH Commun. Math. Comput. Chem.  85  (2021) 47--61.


\end{thebibliography}
\end{document}